\documentclass{amsart}

\input xypic
\usepackage{amssymb}
\usepackage{bbm}
\newtheorem{theorem}{Theorem}[section]

\newtheorem{proposition}[theorem]{Proposition}
\newtheorem{lemma}[theorem]{Lemma}
\newtheorem{corollary}[theorem]{Corollary}

\newtheorem{remark}[theorem]{Remark}

\newcommand{\hlgy}[2]{\ensuremath{H_{#1}(#2)}}
\newcommand{\rhlgy}[2]{\ensuremath{\widetilde{H}_{#1}(#2)}}

\newcounter{bean}

\newenvironment{romanlist}{\begin{list}{\rm ({\roman{bean}})}
      {\usecounter{bean}\setlength{\rightmargin}{\leftmargin}}}
      {\end{list}}
      
\newenvironment{my_enumerate}
{\begin{enumerate}
  \setlength{\itemsep}{0pt}
  \setlength{\parskip}{0pt}
  \setlength{\parsep}{0pt}
  \setlength{\parindent}{0pt}}
{\end{enumerate}}

\newcommand{\inclusion}[2]{\ensuremath{#1
 {\hookrightarrow}#2}}

\newcommand{\mapa}[2]{\ensuremath{#1
 {\longrightarrow}#2}}
\newcommand{\map}[3]{\ensuremath{#1\colon#2
 {\longrightarrow}#3}}
\newcommand{\seqm}[3]{\ensuremath{#1\stackrel{#2}
 {\longrightarrow}#3}}
\newcommand{\seqmm}[5]{\ensuremath{#1\stackrel{#2}
 {\longrightarrow}#3\stackrel{#4}{\longrightarrow}#5}}
\newcommand{\seqmmm}[7]{\ensuremath{#1\stackrel{#2}
 {\longrightarrow}#3\stackrel{#4}{\longrightarrow}#5
  \stackrel{#6}{\longrightarrow}#7}}

\newcommand{\floor}[1]{\ensuremath{\left\lfloor #1 \right\rfloor}}
\newcommand{\paren}[1]{\ensuremath{\left\{ #1 \right\}}}

\newcommand{\qqed}{\hfill\square}

\newcommand{\zmodtwo}{\ensuremath{\mathbb{Z}_{2}}}
\newcommand{\plocal}{\ensuremath{\mathbb{Z}_{(p)}}}

\begin{document}


\title{Homotopy Decompositions of Looped Stiefel manifolds, and their Exponents}
\author{Piotr Beben}
\address{\scriptsize{Department of Mathematics, National University of Singapore
Block S17 (SOC1),
10, Lower Kent Ridge Road,
Singapore 119076}} 
\email{matbpd@nus.edu.sg}

\date{}

\begin{abstract}
Let $p$ be an odd prime, and fix integers $m$ and $n$ such that $0<m<n\leq (p-1)(p-2)$. We give a $p$-local homotopy decomposition for the loop space of the complex Stiefel manifold $W_{n,m}$. Similar decompositions are given for the loop space of the real and symplectic Stiefel manifolds. As an application of these decompositions, we compute upper bounds for the $p$-exponent of $W_{n,m}$. Upper bounds for $p$-exponents in the stable range $2m<n$ and $0<m\leq (p-1)(p-2)$ are computed as well.

\end{abstract}

\maketitle

\section{Introduction}

Fix $p$ to be an odd prime. Throughout this paper we assume that all spaces have been localized at $p$, and we set $q=2(p-1)$. When a reference to the cell structure of a space is made, we will be referring to a given $p$-local cell structure. For $p$-localizations of $CW$-complexes in particular, it will be implicit that the $p$-local cell structure being used is the one induced by localizing.

We shall use the term \textit{fibration} to refer to both homotopy fibrations and fibrations in the strict sense. When stating the homology (or cohomology) of a space without specifying the coefficients, this will be taken to mean that the statement holds for both $\mathbb{Z}_{p}$-homology and $\plocal$-homology.      

Let $W_{n,m}$ be the \textit{complex Stiefel manifold}, the group quotient $SU(n)/SU(n-m)$. Our first theorem provides a nontrivial homotopy decomposition for the loop space of low rank Stiefel manifolds.

\begin{theorem}
\label{MAIN1}
Fix integers $n$ and $m$ such that $0<m<n\leq (p-1)(p-2)$. Then there exists a product decomposition
$$\Omega W_{n,m}\simeq \displaystyle\prod_{1\leq i\leq p-1} \Omega D_{i}$$
\noindent such that for each $1\leq i\leq (p-1)$, $D_{i}$ is an $H$-space whose homology is the exterior algebra
$$\hlgy{*}{D_{i}}\cong\Lambda(y_{2(n-m+i)-1},y_{2(n-m+i)-1+q},...,y_{2(n-m+i)-1+k_{i}q}),$$
\noindent where $k_{i}$ is the largest integer such that $2(n-m+i)-1+k_{i}q\leq 2n-1$.

\end{theorem}

We should also mention that the complex Stiefel manifolds are not $H$-spaces in general (even in the $p$-local sense), so one should not hope that the above decomposition 
will hold before looping. In particular, it is not clear whether the above decomposition is an $H$-space decomposition. 

Theorem~\ref{MAIN1} is based on a decomposition of the unitary group $SU(n)$ for
arbitrary $n$ as a product of indecomposable spaces as in~\cite{MNT2} and~\cite{T2}. Our approach expands
on ideas of Theriault~\cite{T2}, and uses a fundamental construction of Cohen and Neisendorfer~\cite{CN1}
to give a decomposition that is functorial and enjoys good naturality properties (see Theorem~\ref{T2}). The real analog of this decomposition is provided in Theorem~\ref{T3}. In a similar vein to Theorem~\ref{MAIN1}, this leads us to low rank homotopy decompositions for the loop spaces of \textit{real Stiefel manifolds} $V_{n,m}=SO(n)/SO(n-m)$ and \textit{symplectic Stiefel manifolds} $X_{n,m}=Sp(n)/Sp(n-m)$, stated as Theorems~\ref{MAIN2} and~\ref{MAIN3}.

It is clear that these decompositions have an application towards computing $p$-exponents. 
Recall for an arbitrary space $X$ the $p$-exponent $exp_{p}(X)$ of $X$ is defined as the smallest power $p^t$ that annihilates the $p$-primary torsion of $\pi_{*}(X)$. Then we have the following. 

\begin{theorem}
\label{MAIN4}
Fix $0<m\leq (p-1)(p-2)$ and assume either $2m<n$ or $0<m<n\leq (p-1)(p-2)$. Let $k$ be the number of cells in the suspended stunted complex projective space $\Sigma\mathbb{C}P^{n-1}_{m}$ that are in dimensions of the form $(2n-1-iq)$ for $0\leq i<p-1$. Then $$exp_{p}(W_{n,m})\leq p^{n-1+(k-1)}.$$ 

Furthermore, if $k>1$ and $0<m<n\leq (p-1)(p-2)$, and there exists a cell of dimension $(2n-1-iq)$ in $\Sigma\mathbb{C}P^{n-1}_{m}$ such that $i>0$ and $(2n-1-iq)$ is divisible by $p$, then 
$$exp_{p}(W_{n,m})\leq p^{n-1+(k-2)}.$$ 

\end{theorem}

We will see (Remark~\ref{rProd}) that the precise bound $\exp_{p}(W_{n,m})=p^{n-1}$ holds whenever $m\leq p-1$.
By using Theorems \ref{MAIN2} and \ref{MAIN3} we can also compute $p$-exponent bounds for the real and symplectic Stiefel manifolds within certain dimensional ranges. These results are stated as Theorems \ref{MAIN5} and \ref{MAIN6} without proof.

\section{Preliminary Facts About Finite $H$-spaces}

Let $\mathcal{C}$ be the sub-category of spaces and continuous maps defined as follows. The objects in $\mathcal{C}$ are $p$-localizations of path-connected $CW$-complexes $X$, where $X$ consists of no more than $p-2$ odd dimensional cells and no even cells, and the morphisms are continuous maps between these spaces. Let $\mathcal{D}$ be the category of $p$-local finite $H$-spaces spaces and $H$-maps. In this section we recall Cohen and Neisendorfer's~\cite{CN1} construction of a functor between these categories, which will be of fundamental use in our proof of Theorem~\ref{MAIN1}.

\begin{theorem}
\label{T1}
Let $X$ be a space in $\mathcal{C}$. There exists a functor \map{M}{\mathcal{C}}{\mathcal{D}} such that:
\begin{romanlist}
\item \hlgy{*}{M(X)} $\cong$ $\Lambda (\rhlgy{*}{X})$, and there is a functorial map \map{\iota}{X}{M(X)} that induces an inclusion of generating sets on homology;  
\item  there exist functorial maps \seqmm{M(X)}{s}{\Omega\Sigma X}{r}{M(X)} such that the composition $r\circ s$ is homotopic to the identity;
\item  the composition \seqmm{X}{\iota}{M(X)}{s}{\Omega\Sigma X} induces the inclusion 
\seqm{\hlgy{*}{X}}{}{T(\rhlgy{*}{X})} on homology.~$\qqed$
\end{romanlist}

\end{theorem}

The $H$-space structures for the spaces under the image of $M$ are induced by the retraction in part $(ii)$ of Theorem~\ref{T1}.
The functor $M$ takes certain cofibrations to fibrations, as is stated in the following proposition from ~\cite{CN1}. 

\begin{proposition}
\label{T1B}

Let $X$ and $Y$ be spaces in $\mathcal{C}$. Let $X'$ be a $p$-local subcomplex of $X$, and $X''$ be the cofibre of the inclusion \mapa{X'}{X}. 
\begin{romanlist}
\item There exists a fibration $\seqmm{M(X')}{}{M(X)}{}{M(X'')}$;
\item if $X$ and $Y$ have $l$ and $m$ cells such that $l+m\leq p-2$, then $M(X\vee Y)$ is homotopy equivalent to $M(X)\times M(Y)$.~$\qqed$
\end{romanlist}
\end{proposition}

Observe that Theorem \ref{T1} implies $M(S^{2n-1})=S^{2n-1}$. Hence Proposition \ref{T1B} implies a cofibration sequence \seqmm{X'}{}{X}{}{S^{2n-1}} gives a fibration sequence \seqmm{M(X')}{}{M(X)}{}{S^{2n-1}}.

\section{Decomposition of Looped Stiefel Manifolds}

\subsection{Complex Stiefel Manifolds}

The following decomposition of the suspended complex projective space is due to Mimura, Nishida, and Toda ~\cite{MNT1}. 

\begin{lemma}
\label{L1}
For each positive integer $n$, there exists a wedge decomposition 
$$\Sigma\mathbb{C}P^{n-1}\simeq\displaystyle\bigvee_{1\leq i\leq p-1}C_{i}$$
\noindent with
$$\rhlgy{*}{C_{i}}\cong\{x_{2i+1},x_{(2i+1)+q},...,x_{(2i+1)+k_{i}q}\},$$
\noindent where $k_{i}$ is the largest integer such that $(2i+1)+k_{i}q\leq 2n-1$. These decompositions are natural with respect to inclusions $j:$\inclusion{\Sigma\mathbb{C}P^{n-m-1}}{\Sigma\mathbb{C}P^{n-1}}.~$\qqed$
\end{lemma}

The \textit{stunted complex projective space} $\mathbb{C}P^{n}_{m}$ is the cofibre of the inclusion $\seqm{\mathbb{C}P^{n-m}}{j}{\mathbb{C}P^{n}}$.
By the naturality of the above decompositions, $j$ splits as a wedge of maps \seqm{C^{\prime}_i}{j_i}{C_i} for $1\leq i\leq p-1$. Then the cofibre $\Sigma\mathbb{C}P^{n-1}_{m}$ of $j$ splits as a wedge of $p-1$ spaces that are the homotopy cofibres of the maps $j_i$. We record this as the following corollary.

\begin{corollary}
\label{C1}
For each pair of positive integers $m<n$, there exists a wedge decomposition 
$$\Sigma\mathbb{C}P^{n-1}_{m}\simeq\displaystyle\bigvee_{1\leq i\leq p-1}A_{i}$$
\noindent with
$$\rhlgy{*}{A_{i}}\cong\{x_{2(n-m+i)-1},x_{2(n-m+i)-1+q},...,x_{2(n-m+i)-1+k_{i}q}\},$$
\noindent where $k_{i}$ is the largest integer such that $2(n-m+i)-1+k_{i}q\leq 2n-1$.~$\qqed$

\end{corollary}

A decomposition of the unitary group $SU(n)$ for arbitrary $n$ as a product of indecomposable spaces was given by Mimura, Nishida, and Toda in~\cite{MNT2}. The decompositions were of the form $SU(n)\simeq\prod^{p-1}_{i=1}\bar{B}_{i}$ such that \hlgy{*}{\bar{B}_{i}}$\cong\Lambda(x_{2i+1},x_{(2i+1)+q},...,x_{(2i+1)+k_{i}q})$, and $k_{i}$ is the largest integer such that $(2i+1)+k_{i}q\leq 2n-1$. A similar decomposition of $SU(n)$ is given by Theriault~\cite{T2} for $n\leq (p-1)(p-3)$, but this time each of the factors are generated by the functor $M$. We recover Theriault's decomposition for the slightly larger dimensional range $n\leq (p-1)(p-2)$. Along with this, we have the additional property that our decomposition is natural with respect to the inclusion \seqm{SU(n-m)}{\tilde{j}}{SU(n)} of $(n-m)$-frames into $n$-frames. This also presents an advantage over Mimura's, Nishida's, and Toda's decomposition in the sense that their decompositions are not known to be natural. Another advantage is that the maps between corresponding factors in these decompositions fit into certain fibration sequences, as is stated in part~$(iii)$ of Theorem~\ref{T2}. 

\begin{theorem}
\label{T2}
Fix integers $m$ and $n$ such that $0<m<n\leq (p-1)(p-2)$. Then there exists a homotopy commutative diagram of product decompositions
\[\diagram
		\prod^{p-1}_{i=1}B^{\prime}_{i}\dto^{\simeq}\rto^{\prod g_{i}}
		&\prod^{p-1}_{i=1}B_{i}\dto^{\simeq}\\
		SU(n-m)\rto^{\tilde{j}}
		&SU(n),
\enddiagram\]
\noindent such that the following properties hold:

\begin{romanlist}
\item \hlgy{*}{B_{i}}$\cong\Lambda(x_{2i+1},x_{(2i+1)+q},...,x_{(2i+1)+k_{i}q})$, where $k_{i}$ is the largest integer such that $(2i+1)+k_{i}q\leq 2n-1$; 
\item \hlgy{*}{B^{\prime}_{i}}$\cong\Lambda(x_{2i+1},x_{(2i+1)+q},...,x_{(2i+1)+k^{\prime}_{i}q})$, where $k^{\prime}_{i}$ is the largest integer such that $(2i+1)+k^{\prime}_{i}q\leq 2(n-m)-1$; 
\item There exist fibrations \seqmm{B^{'}_{i}}{g_i}{B_{i}}{}{D_{i}}, where $D_{i}$ is an $H$-space such that \hlgy{*}{D_{i}}$\cong\Lambda(x_{(2i+1)+(k^{\prime}_{i}+1)q},...,x_{(2i+1)+k_{i}q})$.

\end{romanlist} 
\end{theorem}

\begin{proof}

Fix an integer $i$ such that $1\leq i\leq p-1$. Let $C^{\prime}_{i}$ and $C_{i}$ be the corresponding summands in the wedge decompositions of $\Sigma\mathbb{C}P^{n-m-1}$ and $\Sigma\mathbb{C}P^{n-1}$ in Lemma \ref{L1}. For each natural number $k$, there exists a map \mapa{\Sigma\mathbb{C}P^{k-1}}{SU(k)} that induces on homology an isomorphism onto the generating set of $\hlgy{*}{SU(k)}\cong\Lambda(\hlgy{*}{\Sigma\mathbb{C}P^{k-1}})$. These maps are natural in the sense that we have the following commutative diagrams
\begin{equation}
\label{D1less}
\diagram
		\Sigma\mathbb{C}P^{n-m-1}\dto^{}\rto^{j}
		&\Sigma\mathbb{C}P^{n-1}\dto^{}\\
		SU(n-m)\rto^{\tilde{j}}
		&SU(n).
\enddiagram
\end{equation}
Take the compositions  $h\colon$\seqmm{C_{i}}{}{\Sigma\mathbb{C}P^{n-1}}{}{SU(n)} and $h^{\prime}\colon$ \seqmm{C^{\prime}_{i}}{}{\Sigma\mathbb{C}P^{n-m-1}}{}{SU(n-m)}. Combining the diagram in (\ref{D1less}) with the naturality of the decompositions in Lemma~\ref{L1}, we have a map \seqm{C^{\prime}_i}{j_{i}}{C_i} such that the following diagram homotopy commutes
\begin{equation}
\label{D1}
\diagram
		C^{\prime}_{i}\dto^{h^{\prime}}\rto^{j_{i}}
		&C_{i}\dto^{h}\\
		SU(n-m)\rto^{\tilde{j}}
		&SU(n).
\enddiagram
\end{equation}
Since $SU(n)$ and $SU(n-m)$ are homotopy associative $H$-spaces, and $\tilde{j}$ is an $H$-map, from the universal property of the James construction we obtain a homotopy commutative diagram
\[\diagram
		\Omega\Sigma C^{\prime}_{i}\dto^{\bar{h}^{\prime}}\rto^{\Omega\Sigma j_{i}}
		&\Omega\Sigma C_{i}\dto^{\bar{h}}\\
		SU(n-m)\rto^{\tilde{j}}
		&SU(n),
\enddiagram\]
\noindent where $\bar{h}$ and $\bar{h}^{\prime}$ are $H$-maps extending the maps $h$ and $h^{\prime}$.
Since $1\leq n\leq (p-1)(p-2)$, the space $C_{i}$ consists of less than $p-1$ odd dimensional cells. Thus we can apply Theorem \ref{T1} to obtain an $H$-space $B_{i}=M(C_{i})$, a map \seqm{C_{i}}{\iota}{B_{i}} that induces an inclusion of generating sets on homology, and a map \seqm{B_{i}}{s}{\Omega\Sigma C_{i}} with a left homotopy inverse. Similarly we obtain an $H$-space $B^{\prime}_{i}=M(C^{\prime}_{i})$, and maps $\iota^{\prime}$ and $s^{\prime}$ with similar properties. The map \seqm{C^{\prime}_{i}}{j_{i}}{C_{i}} induces an $H$-map \seqm{B^{\prime}_{i}}{g_{i}}{B_{i}} via the functor $M$, and we have the following homotopy commutative diagram
\begin{equation}
\label{D2}
\diagram
		B^{\prime}_{i}\dto^{s^{\prime}}\rto^{g_{i}}
		&B_{i}\dto^{s}\\
		\Omega\Sigma C^{\prime}_{i}\dto^{\bar{h}^{\prime}}\rto^{\Omega\Sigma j_{i}}
		&\Omega\Sigma C_{i}\dto^{\bar{h}}\\
		SU(n-m)\rto^{\tilde{j}}
		&SU(n),
\enddiagram
\end{equation}
where the top square commutes because of the functorial property of the maps in Theorem~\ref{T1}~$(ii)$. Using part $(iii)$ of Theorem~\ref{T1}, $\bar{h}\circ s$ induces an inclusion of the generating set of \hlgy{*}{B_{i}} into the generating set of \hlgy{*}{SU(n)}. Similarly $\bar{h}^{\prime}\circ s^{\prime}$ induces an inclusion of the generating set of \hlgy{*}{B^{\prime}_{i}} into the generating set of \hlgy{*}{SU(n-m)}.

Taking the product of diagrams (\ref{D2}) for every integer $i$ such that $1\leq i\leq (p-1)$, we obtain the following homotopy commutative diagram.
\[\diagram
		\prod^{p-1}_{i=1}B^{\prime}_{i}\dto^{\prod f^{\prime}_{i}}\rto^{\prod g_{i}}
		&\prod^{p-1}_{i=1}B_{i}\dto^{\prod f_{i}}\\
		\prod^{p-1}_{i=1}SU(n-m)\dto^{mult.}\rto^{\prod \tilde{j}}
		&\prod^{p-1}_{i=1}SU(n)\dto^{mult.}\\
		SU(n-m)\rto^{\tilde{j}}
		&SU(n),		
\enddiagram\]
\noindent where the left and right vertical compositions induce isomorphisms on the generating sets of the respective homology rings. Dualizing to mod-$p$ cohomology, both vertical compositions induce algebra maps that are isomorphisms on generating sets, so they both induce isomorphisms on mod-$p$ cohomology. Therefore both vertical compositions in the above diagram are homotopy equivalences.     

Finally, for each $1\leq i\leq p-1$, let $\bar{A}_{i}$ be the cofibre of the inclusion \seqm{C^{\prime}_{i}}{j_{i}}{C_{i}}. Then $\bar{A}_{i}$ consists of no more than $p-2$ odd dimensional cells. Applying Proposition \ref{T1B} to the cofibration sequence \seqmm{C^{\prime}_{i}}{j}{C_{i}}{}{\bar{A}_i} for each integer $i$, we obtain fibration sequences
$$\seqmm{B^{\prime}_{i}}{g_{i}}{B_{i}}{}{D_{i}},$$
where $D_{i}=M(\bar{A}_{i})$.\end{proof}

\begin{remark}
\label{R1}

Notice that the spaces $\bar{A}_{i}$ such that $D_{i}=M(\bar{A}_{i})$ are (with indices rearranged) precisely the summands in the wedge decomposition of $\Sigma\mathbb{C}P^{n-1}_{m}$ in Corollary \ref{C1}.

\end{remark}

We now prove one of our main theorems.

\begin{proof}[Proof of Theorem \ref{MAIN1}]

Applying Theorem \ref{T2} we obtain a diagram of fibration sequences
\begin{equation}
\label{D6}
\diagram
		\prod^{p-1}_{i=1}\Omega D_{i}\dto^{\ell}\rto^{}
		&\prod^{p-1}_{i=1}B^{\prime}_{i}\dto^{\simeq}\rto^{\prod g_{i}}
		&\prod^{p-1}_{i=1}B_{i}\dto^{\simeq}\\
		\Omega W_{n,m}\rto^{}
		&SU(n-m)\rto^{\tilde{j}}
		&SU(n),
\enddiagram
\end{equation}
\noindent for some induced map of fibres $\ell$. This diagram implies that the map $\ell$ is a homotopy equivalence by the $5$-lemma.

\end{proof}

\subsection{Real Stiefel Manifolds}
\label{sReal}

Where localized at an odd prime $p$, there is a difference in the homology of $SO(n)$ when $n$ is even as opposed to odd. That is, we have homology isomorphisms
\begin{equation}
\label{eHlgy1}
\hlgy{*}{SO(2k+1)}\cong\Lambda(x_{3},x_{7},...,x_{4k-1})
\end{equation}
\noindent and
\begin{equation}
\label{eHlgy2}
\hlgy{*}{SO(2k)}\cong\Lambda(x_{3},x_{7},...,x_{4k-5},\bar{x}_{2k-1}). 
\end{equation}
\noindent The inclusion of $(n-m)$-frames into $n$-frames \seqm{SO(n-m)}{\tilde{j}}{SO(n)} induces on homology the algebra map that sends each generator $x_{i}\in\hlgy{*}{SO(n-m)}$ to the corresponding generator $x_{i}\in\hlgy{*}{SO(n)}$, and if $n-m$ is even, the generator $\bar{x}_{n-m-1}\in\hlgy{*}{SO(n-m)}$ is mapped trivially.

For $n=2k$ it is well known (Theorem $6.5$ in reference~\cite{MT}) that there exists a decomposition 
\begin{equation}
\label{eDecomp}
SO(2k)\simeq S^{2k-1}\times SO(2k-1).
\end{equation}

\noindent Harris \cite{Harris} showed there are decompositions
\begin{equation}
\label{eHarris}
SU(2k)\simeq SO(2k+1)\times (SU(2k)/Sp(k))
\end{equation} 

\noindent that are natural with respect to the inclusions \seqm{SO(2(k-k')+1)}{\tilde{j}}{SO(2k+1)} and \seqm{SU(2(k-k'))}{\tilde{j}}{SU(2k)} for $k'\leq k$. With this we can prove the following homotopy decomposition as an application of Theorem \ref{T2}. A general form of this decomposition was found by Mimura, Nishida, and Toda~\cite{MNT2}, but The same advantages hold in our decomposition as was the case for the special unitary groups $SU(n)$ in the previous section.

\begin{theorem}
\label{T3}
Fix integers $m$ and $n$ such that $0<m<n\leq (p-1)(p-2)+1$, and let $r=\floor{\frac{p-1}{2}}$. Then there exists a homotopy commutative diagram of product decompositions
\begin{equation}
\label{eSquare}
\diagram
		X^{\prime}\times\prod^{r}_{i=1}\mathcal{B}^{\prime}_{i}\dto^{\simeq}\rto^{\bar{g}\times\prod g_{i}}
		&X\times\prod^{r}_{i=1}\mathcal{B}_{i}\dto^{\simeq}\\
		SO(n-m)\rto^{\tilde{j}}
		&SO(n)
\enddiagram
\end{equation}
\noindent such that the following properties hold.

\begin{romanlist}
\item \hlgy{*}{\mathcal{B}_{i}}$\cong\Lambda(x_{2i+1},x_{(2i+1)+2q},...,x_{(2i+1)+2k_{i}q})$, where $k_{i}$ is the largest integer such that $(2i+1)+2k_{i}q\leq 2n-3$; 
\item \hlgy{*}{\mathcal{B}^{\prime}_{i}}$\cong\Lambda(x_{2i+1},x_{(2i+1)+2q},...,x_{(2i+1)+2k^{\prime}_{i}q})$, where $k^{\prime}_{i}$ is the largest integer such that $(2i+1)+2k^{\prime}_{i}q\leq 2(n-m)-3$; 
\item There exist fibrations \seqmm{\mathcal{B}^{\prime}_{i}}{f_i}{\mathcal{B}_{i}}{}{\mathcal{D}_{i}}, $\hlgy{*}{\mathcal{D}_{i}}\cong\Lambda(x_{(2i+1)+2(k^{\prime}_{i}+1)q},...,x_{(2i+1)+2k_{i}q})$, 
and $\mathcal{D}_{i}$ is an $H$-space; 
\item The map \seqm{X^{\prime}}{\bar{g}}{X} is the trivial map;
\item If $n-m$ is even, $X^{\prime}=S^{n-m-1}$, and if $n-m$ is odd, then $X^{\prime}$ is a point;
\item If $n$ is even, $X=S^{n-1}$, and if $n$ is odd, then $X$ is a point.

\end{romanlist} 
\end{theorem}
\begin{proof}

Throughout this proof let us fix $n$ and $n-m$ both odd, and $0<m<n\leq (p-1)(p-2)$. Let $r=\floor{\frac{p-1}{2}}$. Recall from Theorem \ref{T2} the decompositions of the special unitary groups $SU(n-1)$ and $SU(n-m-1)$ - as products of $p-1$ factors $B^{\prime}_{i}$ and $B_{i}$ respectively - and recall the homology of each of the factors in these decompositions. Restricting to the odd factors, we have a homotopy commutative square 
\[\diagram
		\prod^{r}_{i=1}B^{\prime}_{2i-1}\dto^{}\rto^{\prod g_{2i-1}}
		&\prod^{r}_{i=1}B_{2i-1}\dto^{}\\
		SU(n-m-1)\rto^{\tilde{j}}
		&SU(n-1).
\enddiagram\]
\noindent Since $n$ and $n-m$ are odd, $SO(n-m)$ and $SO(n)$ are retracts of $SU(n-m-1)$ and $SU(n-1)$. The naturality of this retraction implies we have the following homotopy commutative square
\begin{equation}
\label{eD1}
\diagram
		\prod^{r}_{i=1}B^{\prime}_{2i-1}\dto^{\simeq}_{\phi'}\rto^{\prod g_{2i-1}}
		&\prod^{r}_{i=1}B_{2i-1}\dto^{\simeq}_{\phi}\\
		SO(n-m)\rto^{\tilde{j}}
		&SO(n),
\enddiagram
\end{equation}
\noindent where we observe that the vertical maps induce isomorphisms on homology, so they are homotopy equivalences. We complete the proof for $n$ and $n-m$ both odd by setting $\mathcal{B}_{i}=B_{2i-1}$, $\mathcal{B}^{\prime}_{i}=B_{2i-1}$, $f_{i}=g_{2i-1}$, $\mathcal{D}_{i}=D_{2i-1}$, and applying Theorem \ref{T2}.

To complete the proof for the other cases, we keep $n$ and $n-m$ odd. For convenience set $f=\prod^{r}_{i=1} f_{i}$, $A=\prod^{r}_{i=1}\mathcal{B}_{i}$, and $A^{\prime}=\prod^{r}_{i=1}\mathcal{B}^{\prime}_{i}$. On homology \seqm{SO(n-m+1)}{\tilde{j}}{SO(n-m+2)} sends the generator $\bar{x}_{n-m}$ trivially, so the homology Serre exact sequence for the fibration sequence
$\seqmmm{\Omega S^{n-m+1}}{\delta}{SO(n-m+1)}{\tilde{j}}{SO(n-m+2)}{\pi}{S^{n-m+1}}$
implies $\delta_{*}$ sends the bottom generator of \hlgy{*}{\Omega S^{n-m+1}} to $c\cdot\bar{x}_{n-m}$ for some integer $c$ prime to $p$. Thus the Hurewicz image of the composition
$\iota\colon\seqmm{S^{n-m}}{E}{\Omega S^{n-m+1}}{\delta}{SO(n-m+1)}$
is $c\cdot\bar{x}_{n-m}$. By exactness of the homotopy long exact sequence $\tilde{j}\circ\iota$ is null homotopic, implying 
the composition \seqmm{S^{n-m}}{\iota}{SO(n-m+1)}{\tilde{j}}{SO(n)} is also nullhomotopic.
Since \seqm{A^{\prime}}{\phi'}{SO(n-m)} is a homotopy equivalence, \seqm{SO(n-m)}{\tilde{j}}{SO(n-m+1)} induces an inclusion of algebras on homology, and the Hurewicz image of $\iota_{*}$ is $c\cdot\bar{x}_{n-m}$, then the composition
$\theta'=\iota\cdot(\tilde{j}\circ\phi')\colon\seqm{S^{n-m}\times A^{\prime}}{}{SO(n-m+1)}$
induces an isomorphism on homology, and so it is a homotopy equivalence. Similarly we have a map \seqm{S^{n}}{\iota}{SO(n+1)} whose Hurewicz image is $d\cdot\bar{x}_{n-m}$ for some $d$ prime to $p$. Thus 
$\theta=\iota\cdot(\tilde{j}\circ\phi)\colon\seqm{S^{n}\times A}{}{SO(n+1)}$
is a homotopy equivalence. 

Taking products we obtain the following homotopy commutative diagram
\begin{equation}
\label{eD4}
\diagram
		S^{n-m}\times A^{\prime}\dto^{\iota\times\phi'}\rto^{*\times f}
		&*\times A\dto^{*\times\phi}\\
		(SO(n-m+1))^{2}\dto^{mult.}\rto^(0.6){\tilde{j}\times\tilde{j}}
		&(SO(n))^{2}\dto^{mult.}\\
		SO(n-m+1)\rto^(0.6){\tilde{j}}
		&SO(n),
\enddiagram
\end{equation} 
\noindent where the bottom square commutes since $\tilde{j}$ is an $H$-map. Consider the following diagram
\begin{equation}
\label{eD5}
\diagram
		A^{\prime}\dto^{\simeq}_{\phi'}\rto^{*\times\mathbbm{1}}
		&S^{n-m}\times A^{\prime}\dto^{\simeq}_{\theta'}\rto^(0.6){*\times f}
		&A\dto^{\simeq}_{\phi}\rto^(0.4){*\times\mathbbm{1}}
		&S^{n}\times A\dto^{\simeq}_{\theta}\\
		SO(n-m)\rto^{\tilde{j}}
		&SO(n-m+1)\rto^(0.6){\tilde{j}}
		&SO(n)\rto^(0.4){\tilde{j}}
		&SO(n+1).
\enddiagram
\end{equation} 
The proof will be complete if this diagram homotopy commutes. Here the left and right squares homotopy commute by the construction of $\theta$ and $\theta'$, and the middle square is the outer part of the diagram in (\ref{eD4}).
\end{proof}

Theorem \ref{T3} allows us to decompose the loop spaces of low rank real Stiefel manifolds $V_{n,m}=O(n)/O(n-m)$ as follows.

\begin{theorem}
\label{MAIN2}
Fix integers $n$ and $m$ such that $0<m<n\leq(p-1)(p-2)+1$. Let $r=\floor{\frac{p-1}{2}}$. Then there exists a product decomposition
$$\Omega V_{n,m}\simeq X^{\prime}\times\Omega X\times\displaystyle\prod^{r}_{i=1} \Omega\mathcal{D}_{i}$$
where each $\mathcal{D}_{i}$ is the $H$-space $D_{2i-1}$ from Theorem \ref{T3}, and 
\begin{romanlist}
\item If $n-m$ is even, $X^{\prime}=S^{n-m-1}$, and if $n-m$ is odd, then $X^{\prime}$ is a point;
\item If $n$ is even, $X=S^{n-1}$, and if $n$ is odd, then $X$ is a point.
\end{romanlist}

\end{theorem}

\begin{proof}

The proof is similar to that of Theorem \ref{MAIN1}. Applying the diagram in (\ref{eSquare}) from Theorem \ref{T3}, and noting that $\bar{g}$ is the trivial map, we obtain a diagram of fibration sequences
\[\diagram
		X^{\prime}\times\Omega X\times\prod^{r}_{i=1}\Omega\mathcal{D}_{i}\dto^{\ell}\rto^{}
		&X^{\prime}\times\prod^{r}_{i=1}\mathcal{B}^{\prime}_{i}\dto^{\simeq}\rto^{\bar{g}\times\prod g_{i}}
		&X\times\prod^{r}_{i=1}\mathcal{B}_{i}\dto^{\simeq}\\
		\Omega V_{n,m}\rto^{}
		&SO(n-m)\rto^{\tilde{j}}
		&SO(n)
\enddiagram\]
\noindent for some induced map of fibres $\ell$. Since the middle and right vertical maps are homotopy equivalences, the map $\ell$ is a homotopy equivalence by the $5$-lemma. Finally as we saw in the proof of Theorem \ref{T3}, $\mathcal{D}_{i}=D_{2i-1}$ where each $D_{2i-1}$ is one the $H$-spaces from Theorem~\ref{MAIN1}. 

\end{proof}

\subsection{Symplectic Stiefel Manifolds}
\label{sSymp}

Harris \cite{Harris} showed that localized at odd primes $p$, there is a natural homotopy equivalence 
$$Sp(n)\simeq Spin(2n+1),$$ 
where the \textit{spinor group} $Spin(2n+1)$ is the simply connected cover of $SO(2n+1)$. Since (integrally) we have $\pi_{1}(SO(2n+1))\cong \zmodtwo$, then $\pi_{1}(SO(2n+1))=0$ when localized at an odd prime $p$. Thus there is a natural $p$-local homotopy equivalence 
$$Spin(2n+1)\simeq SO(2n+1).$$ 

\noindent With this information we can use Theorem \ref{T3} to decompose $Sp(n)$ when $n<\frac{1}{2}(p-1)(p-2)$. In a similar manner as before we decompose the loop spaces of low rank symplectic Stiefel manifolds $X_{n,m}=Sp(n)/Sp(n-m)$. This is stated as follows.

\begin{theorem}
\label{MAIN3}
Fix integers $k$ and $j$ such that $0<j<k\leq\frac{1}{2}(p-1)(p-2)$. Let $r=\floor{\frac{p-1}{2}}$. Then there exists a product decomposition
$$\Omega X_{k,j}\simeq \displaystyle\prod^{r}_{i=1} \Omega\mathcal{D}_{i}$$
\noindent where each $\mathcal{D}_{i}$ is the $H$-space $D_{2i-1}$ from Theorem \ref{T3}, for $n=2k+1$ and $m=2j$. 
$~\qqed$
\end{theorem}

\section{Exponents}

As an application of our decompositions of $\Omega W_{n,m}$ we compute upper bounds for the $p$-exponents of $W_{n,m}$ in the range $0<m<n\leq (p-1)(p-2)$. The $p$-exponents in the stable range $0<m\leq (p-1)(p-2)$ and $2m<n$ will also be considered, though using different methods. 

Recall that the integral James number $U(n,m)$ of $W_{n,m}$ is defined as the degree of the map \mapa{\mathbbm{Z}}{\mathbbm{Z}} induced by the projection \mapa{W_{n,m}}{W_{n,1}=S^{2n-1}} on $\pi_{2n-1}$, and the $p$-local James number $U_{(p)}(n,m)$ is the $p$-component of $U(n,m)$. The proof of part $(1)$ of the following proposition can be found in Proposition $(7.2)$ of~\cite{Beben}, and Proposition $(6.3)$ of~\cite{MNT2}. Part $(2)$ is an easy consequence of part $(1)$, and can be found in Theorem $(7.1)$ of~\cite{Beben}, or with the use of $K$-theory in~\cite{Crabb}.  

\begin{proposition}
\label{tJN}
Let the space $A$ be a summand in the splitting of a suspended stunted complex projective space in Corollary \ref{C1}. Suppose $A$ has $l<p-1$ cells, with the bottom cell in dimension $2r+1$, and hence the top cell in dimension $2r+1+(l-1)q$. Let $J$ be the unique integer in the range $0\leq J\leq p-1$ such that $r+J(p-1)$ is divisible by $p$, and take the map \seqm{M(A)}{\tilde{\nu}}{S^{2r+1+(l-1)q}} induced by the quotient \seqm{A}{\nu}{S^{2r+1+(l-1)q}}. 
\begin{my_enumerate}
\item If $l-1\leq J$, then $\tilde{\nu}$ induces a degree $p^{l-1}$ from \mapa{\plocal}{\plocal} on $\pi_{2r+1+(l-1)q}$. Otherwise if $l-1=J+1$, then $\tilde{\nu}$ induces a degree $p^{t}$ for some integer $0\leq t\leq l-2$, and if $l-1>J+1$, then $\tilde{\nu}$ induces a degree $p^{t}$ for some integer $1\leq t\leq l-2$.   

\item Fix $0<m\leq (p-1)(p-2)$ and assume either $2m<n$ or $0<m<n\leq (p-1)(p-2)$. Pick $A$ to be the summand of $\Sigma\mathbb{C}P^{n-1}_{m}$ that has its top cell in dimension $2n-1$. Then the degree of $\tilde{\nu}_{*}$ on $\pi_{2n-1}$ is equal to the $p$-local James number $U_{(p)}(n,m)$. 

Cconsequently, whenever there exists a cell of dimension $(2n-1-iq)$ in $\Sigma\mathbb{C}P^{n-1}_{m}$ such that $i>0$ and~$(2n-1-iq)$ is divisible by $p$, then $U_{(p)}(n,m)\leq p^{l-2}$. Otherwise $U_{(p)}(n,m)= p^{l-1}$.
\end{my_enumerate}~$\qqed$
\end{proposition}

We use the following proposition, proven in~\cite{T4}.

\begin{proposition}
\label{MV}
Take a fibration \seqmm{F}{i}{E}{r}{B} with $r$ an $H$-map between the $H$-spaces $E$ and $B$. Suppose there exists a map \map{s}{B}{E} such that the composition \map{r\circ s}{B}{B} is a $p^{t}$-power map for some integer $t$. Then there exists a fibration
$$\seqmm{B\{t\}}{}{F\times B}{}{E},$$
\noindent where $B\{t\}$ is the homotopy fibre of the $p^{t}$-power map \map{r\circ s}{B}{B}.$~\qqed$ 

\end{proposition}

The following lemma will be used to prove part of Theorem \ref{MAIN4}. 

\begin{lemma}
\label{E1}
Let $A$ be a summand in the wedge decomposition of $\Sigma\mathbb{C}P^{n-1}_{m}$ in Corollary \ref{C1}, and let $J$ be the unique integer in the range $1\leq J\leq p$ such that $r+J(p-1)$ is divisible by $p$. If $l-1\leq J$, then $exp_{p}(M(A))\leq p^{r+(l-1)p}$. Otherwise if $l-1>J$, then $exp_{p}(M(A))\leq p^{r+(l-1)p-1}$.
\end{lemma}

\begin{proof}
We shall use $exp_{p}(S^{2k+1})=p^{k}$ and $exp_{p}(S^{2k+1}\{p^t\})=p^{t}$ (Cohen, Moore, and Neisendorfer~\cite{CMN,N3}), which holds for odd primes $p$ and all integers $k\geq 0$. 

Suppose $l-1\leq J$. Fix some $k\leq l-1$ and let $A^{k}$ denote the $(2r+1+kq)$-skeleton of $A=A^{l-1}$. We proceed by induction by assuming that $exp_{p}(M(A^{k-1}))\leq p^{r+(k-1)p}$. The base case $k=1$ holds since $M(A^{0})=M(S^{2r+1})=S^{2r+1}$. For the induction step, note that because $A$ is a summand in the wedge decomposition of a suspended stunted complex projective space, so is its skeleton $A^{k}$. Then by Proposition \ref{tJN} we have a map \seqm{S^{2r+1+kq}}{\alpha}{M(A^{k})} such that the composition \seqmm{S^{2r+1+kq}}{\alpha}{M(A^{k})}{\tilde{\nu}}{S^{2r+1+kq}} is a degree $p^{k}$ map, where $\tilde{\nu}$ is induced by the quotient \seqm{A^{k}}{\nu}{S^{2r+1+kq}}. Since we are localizing at an odd prime $p$, then $S^{2r+1+kq}$ is an $H$-space, and so this composition is also a $p^{k}$-power map. Applying Proposition \ref{MV} to the fibration \seqmm{M(A^{k-1})}{}{M(A^{k})}{\tilde{\nu}}{S^{2r+1+kq}}, there is the following fibration.    
$$\seqmm{S^{2r+1+kq}\{p^{k}\}}{}{M(A^{k-1})\times S^{2r+1+kq}}{}{M(A^{k})}.$$
\noindent So by the homotopy long exact sequence for this fibration and our inductive assumption
\begin{align*}
exp_{p}(M(A^{k}))\leq &
exp_{p}(S^{2r+1+kq}\{p^{k}\})\cdot max(exp_{p}(M(A^{k-1})),exp_{p}(S^{2r+1+kq}))\\
\leq &p^{k}\cdot max(p^{r+(k-1)p},p^{r+k(p-1)})\\
= &p^{k}\cdot p^{r+k(p-1)} = p^{r+kp},
\end{align*}
\noindent where $max(p^{r+(k-1)p},p^{r+k(p-1)})=p^{r+k(p-1)}$ since we assume $k\leq l-1<p-1$. Hence $exp_{p}(M(A))\leq p^{r+(l-1)p}$.

For the case $l-1>J$, the induction starts at the base case $k=J$, where we have shown that $exp_{p}(M(A^{J}))\leq p^{r+Jp}$. If $J<k\leq l-1$, then by Theorem \ref{tJN} we have a map $\alpha$ such that the composition \seqmm{S^{2r+1+kq}}{\alpha}{M(A^{k})}{\tilde{\nu}}{S^{2r+1+kq}} is a $p^{k-1}$-power map. The rest of the induction is the same as the previous case. 
\end{proof}

Even though we failed to obtain analogous decompositions of $\Omega W_{n,m}$ for most choices of $n$ and $m$ in the stable range $m\leq (p-1)(p-2)$ and $2m<n$, fortunately there is a work-around. Together with Lemma~\ref{tJN}, the following lemma allows us to calculate $p$-exponent bounds in this stable range. The results are similar to what could be achieved if such decompositions in reality existed:

\begin{lemma}
\label{lE2}
Fix $p-1<m\leq (p-1)(p-2)$ and $2m<n$. Let \seqm{W_{n,m}}{\pi}{W_{n,p-1}} be the projection map. Then there exists a space $B$, a map \seqm{\Omega B}{\alpha}{\Omega W_{n,m}}, and a homotopy equivalence \seqm{\Omega W_{n,p-1}}{h}{\Omega B} such that the composition 
\seqmmm{\Omega B}{\alpha}{\Omega W_{n,m}}{\Omega\pi}{\Omega W_{n,p-1}}{h}{\Omega B} is a $p^{t}$-power map, and $p^{t}$ is equal to the maximum of the set of James numbers $\paren{U_{(p)}(n-i,m-i)\,|\,0\leq i<p-1}$. 

\end{lemma}

\begin{proof}

We have the following homotopy commutative diagram
\[\diagram
		\bigvee_{i=0}^{p-2}A_{i}\dto^{\simeq}\rto^{\vee q_{i}}
		&\bigvee_{i=0}^{p-2}S^{2n-1-2i}\dto^{\simeq}\\
		\Sigma\mathbb{C}P^{n-1}_{m}\dto^{}\rto^{q}
		&\Sigma\mathbb{C}P^{n-1}_{p-1}\dto^{}\\
		W_{n,m}\rto^{\pi}
		&W_{n,p-1},
\enddiagram\]
\noindent where the vertical homotopy equivalences are due to Corollary \ref{C1} (and we index so that $A_{i}$ has the $(2n-1-2i)$-cell in its top dimension), and the top vertical maps \seqm{A_{i}}{q_{i}}{S^{2n-1-2i}} in the wedge are the quotient maps. 
Using the Hilton-Milnor theorem, $\prod_{i}\Omega A_{i}$ and $\prod_{i}\Omega S^{2n-1-2i}$ are retracts of 
$\Omega(\bigvee_{i}\Omega A_{i})$ and $\Omega(\prod_{i}\Omega S^{2n-1-2i})$, and these retractions are natural with respect to the map $\Omega\vee q_{i}$ (restricting to $\prod \Omega q_{i}$). Thus looping the above diagram one obtains     
\[\diagram
		\prod_{i=0}^{p-2}\Omega A_{i}\dto^{}\rto^{\prod \Omega q_{i}}
		&\prod_{i=0}^{p-2}\Omega S^{2n-1-2i}\dto^{\simeq}\\
		\Omega W_{n,m}\rto^{\Omega\pi}
		&\Omega W_{n,p-1}.
\enddiagram\]

In the stable range $p-1<m\leq (p-1)(p-2)$ and $2m<n$, the second part of Proposition \ref{tJN} implies the multiplication induced by each \seqm{A_{i}}{q_{i}}{S^{2n-1-2i}} on $\pi_{2n-1-2i}$ is equal to the multiplication induced by the projection \seqm{W_{n-i,m-i}}{\pi}{S^{2n-1-2i}}. Hence for each integer $0\leq i<p-1$ we have maps \seqm{S^{2n-1-2i}}{\beta_{i}}{A_{i}} such that each composition \seqmm{S^{2n-1-2i}}{\beta_{i}}{A_{i}}{q_{i}}{S^{2n-1-2i}} is a degree $p^t$ map, where $p^{t}=\max\paren{U_{(p)}(n-i,m-i)\,|\,0\leq i<p-1}$. Since odd spheres are $p$-local $H$-spaces, the loopings of these compositions are $p^t$-power maps. The lemma follows by setting $B=\prod_{0\leq i<p-1}S^{2n-1-2i}$. \end{proof}

\begin{remark}
\label{rProd}
In the proof of Lemma \ref{lE2} we showed $\Omega W_{n,p-1}\simeq \prod_{0\leq i<p-1}\Omega S^{2n-1-2i}$.  
With a similar argument one can show $\Omega W_{n,m}\simeq \prod_{0\leq i<m}\Omega S^{2n-1-2i}$ when $m\leq p-1$, which reproduces
a specific case of a more general result due to Kumpel~\cite{Kumpel}. Thus $\exp_{p}(W_{n,m})=p^{n-1}$ when $m\leq p-1$. 
\end{remark}

We now prove Theorem \ref{MAIN4}.

\begin{proof}[Proof of Theorem \ref{MAIN4}]
Let us first consider the case $0<m<n\leq (p-1)(p-2)$. By Theorem \ref{MAIN1} and Remark \ref{R1} we have the product decomposition $\Omega W_{n,m}\simeq\prod^{p-1}_{i=1}\Omega M(A_{i})$, where each $A_{i}$ is a summand in the wedge decomposition of $\Sigma\mathbb{C}P^{n-1}_{m}$ in Corollary \ref{C1}, and we index so that $A_{i}$ has the $(2(n-m+i)-1)$-cell in its bottom dimension when $i\leq m$, and is trivial if $i>m$. Therefore 
$$exp_{p}(W_{n,m})= max\{exp_{p}(M(A_{i}))|1\leq i\leq p-1\}.$$ 
\noindent Let $t_{i}$ be the number of cells in $A_{i}$. By Lemma \ref{E1} we have the exponent bounds
\begin{equation}
\label{Bound}
exp_{p}(M(A_{i}))\leq p^{n-m+i-1+(t_{i}-1)p}. 
\end{equation}
\noindent We see that this exponent bound is the greatest when $j$ is the integer such that $A_{j}$ has the $(2n-1)$-cell in its top dimension. Therefore $exp_{p}(W_{n,m})\leq p^{n-1+(t_{j}-1)}$. Note that $t_{j}=k$, where $k$ is the number of cells in $\Sigma\mathbb{C}P^{n-1}_{m}$ that are in dimensions of the form $(2n-1-iq)$ for $0\leq i<p-1$. Hence $exp_{p}(W_{n,m})\leq p^{n-1+(k-1)}$.

When $A_{j}$ has a cell in a dimension divisible by $p$, then Lemma~\ref{E1} implies the bound can be improved to $exp_{p}(M(A_{j}))\leq p^{n-1+(t_{j}-2)}$. Still this bound is at least as large as all the bounds in (\ref{Bound}) for $i\neq j$, though possibly no longer strictly as large. Therefore $exp_{p}(W_{n,m})\leq p^{n-1+(t_{j}-2)}=p^{n-1+(k-2)}$ in this case.

For the last case take $2m<n$ and $0<m\leq (p-1)(p-2)$. If $m\leq p-1$, then by Remark \ref{rProd}$\exp_p(W_{n,m})=p^{n-1}$ and we are done. So let us assume $m>p-1$. Note there exists a fibration 
$$\seqmm{W_{n-(p-1),m-(p-1)}}{}{W_{n,m}}{\pi}{W_{n,p-1}}.$$ 
\noindent By Lemma \ref{lE2} there is a space $B$ and a homotopy equivalence \seqm{\Omega W_{n,p-1}}{h}{\Omega B} such that the composition 
$$\seqmmm{\Omega B}{\alpha}{\Omega W_{n,m}}{\Omega\pi}{\Omega W_{n,p-1}}{h}{\Omega B}$$ 
\noindent is a $p^{t}$-power map, and $p^{t}$ is equal to the maximum of the set of James numbers 
$$\paren{U_{(p)}(n-j,m-j)\,|\,0\leq j<p-1}.$$ 
\noindent Since $2m<n$, then $2(m-j)<n-j$, and an upper bound for each of the James numbers in this set are known by Theorem \ref{tJN}. That is,
\begin{equation}
\label{setOfBounds}
U_{(p)}(n-j,m-j)\leq p^{t_{j}-1}
\end{equation}
\noindent where $t_{j}$ is the number of cells in $\Sigma\mathbb{C}P^{n-1-j}_{m-j}$ for dimensions of the form $(2(n-j)-1-iq)$. Therefore the maximum of the bounds in (\ref{setOfBounds}) happens when $j=0$, implying $p^{t}\leq p^{t_{0}-1}$.

Now take the following homotopy commutative diagram of homotopy fibrations 
\begin{equation}
\label{presidentsChoice}
\diagram
		W_{n-(p-1),m-(p-1)}\rto^{}\dto^{\ell}
		&W_{n,m}\ddouble\rto^{\pi}
		&W_{n,p-1}\dto^{h}\\
		F\rto^{}
		&W_{n,m}\rto^{f}
		&B
\enddiagram
\end{equation}
\noindent where the map $f$ is the composition $h\circ\pi$, and $F$ is the homotopy fibre of $f$. Since the middle and right vertical maps are homotopy equivalences, the lift $\ell$ is also a homotopy equivalence by the $5$-lemma. Now applying Proposition \ref{MV} to the bottom fibration, and using the homotopy equivalences in (\ref{presidentsChoice}), we obtain the bound
\begin{equation}
\label{expbound2}
\exp_{p}(W_{n,m})\leq p^{t_{0}-1}\cdot \max(\exp_{p}(W_{n-(p-1),m-(p-1)}), \exp_{p}(W_{n,p-1})).
\end{equation}

Repeat the above argument to get bounds 
\begin{equation}
\label{expbound3}
\exp_{p}(W_{n-j(p-1),m-j(p-1)})\leq p^{t_{0,j}-1}\cdot \max(\exp_{p}(W_{n-(j+1)(p-1),m-(j+1)(p-1)}), \exp_{p}(W_{n-j(p-1),p-1}))
\end{equation}
\noindent where $m-(j+1)(p-1)>0$ and $t_{0,j}$ is the number of cells in $\Sigma\mathbb{C}P^{n-1-j(p-1)}_{m-j(p-1)}$ in dimensions of the form $(2(n-j(p-1))-1-iq)$. Note $t_{0,0}=t_{0}$ and $t_{0,j+1}<t_{0,j}=t_{0,j+1}+1$. By Remark \ref{rProd} we have
$$\exp_{p}(W_{n-j(p-1),p-1})=p^{n-1-j(p-1)}.$$

\noindent We induct on the bound in (\ref{expbound3}) starting with the base case $j=t_{0}-1$, where $0<m-(t_{0}-1)(p-1)\leq p-1$, and then apply Remark \ref{rProd}. The inductive assumption is 
$$\exp_{p}(W_{n-(j+1)(p-1),m-(j+1)(p-1)})\leq p^{n-1-(j+1)(p-1)+(t_{0,j+1}-1)}.$$ 

\noindent Since $0<m\leq (p-1)(p-2)$, $t_{0,j+1}<t_{0}\leq p-2$, and so 
$$\exp_{p}(W_{n-(j+1)(p-1),m-(j+1)(p-1)})< \exp_{p}(W_{n-j(p-1),p-1}).$$ 

\noindent Then using the bound in (\ref{expbound3})
$$\exp_{p}(W_{n-j(p-1),m-j(p-1)})\leq p^{t_{0,j}-1}\cdot \exp_{p}(W_{n-j(p-1),p-1})=p^{n-1-j(p-1)+(t_{0,j}-1)}.$$ 

\noindent Therefore by induction
$$\exp_{p}(W_{n,m})\leq p^{n-1+(t_{0}-1)}.$$ 

\end{proof}

We finish off by giving analogous exponent bounds for real and symplectic Stiefel manifolds. These follow from the decompositions in Theorems~\ref{MAIN2} and~\ref{MAIN3}, and the same argument used to prove Theorem~\ref{MAIN4}.

\begin{theorem}
\label{MAIN5}
Fix $0<m<n\leq (p-1)(p-2)+1$ and let $n$ be odd. Let $k$ be the number of cells in $\Sigma\mathbb{C}P^{n-2}_{m}$ that are in dimensions of the form $(2n-3-iq)$ for $0\leq i<p-1$. Then
$$exp_{p}(V_{n,m})\leq p^{n-2+(k-1)}$$
\noindent and
$$exp_{p}(V_{n+1,m})\leq p^{n-2+(k-1)}.$$

Furthermore, if $k>1$ and there exists a cell of dimension $(2n-3-iq)$ in $\Sigma\mathbb{C}P^{n-2}_{m}$ such that $i>0$ and $(2n-3-iq)$ is divisible by $p$, then 
$$exp_{p}(V_{n,m})\leq p^{n-2+(k-2)}$$
\noindent and
$$exp_{p}(V_{n+1,m})\leq p^{n-2+(k-2)}.$$
$~\qqed$
\end{theorem}

\begin{theorem}
\label{MAIN6}
Fix $0<j<k\leq \frac{1}{2}(p-1)(p-2)$ and let $n=2k+1$ and $m=2j$. Let $k$ be the number of cells in $\Sigma\mathbb{C}P^{n-2}_{m}$ that are in dimensions of the form $(2n-3-iq)$ for $0\leq i<p-1$. Then
$$exp_{p}(X_{k,j})\leq p^{n-2+(k-1)}.$$

Furthermore, if $k>1$ and there exists a cell of dimension $(2n-3-iq)$ in $\Sigma\mathbb{C}P^{n-2}_{m}$ such that $i>0$ and $(2n-3-iq)$ is divisible by $p$, then 
$$exp_{p}(X_{k,j})\leq p^{n-2+(k-2)}.$$
$~\qqed$
\end{theorem}

\bibliographystyle{amsplain}
\bibliography{stiefel-exp}

\end{document}